\documentclass[10pt]{article} 
 
\usepackage[utf8]{inputenc} 

\usepackage{geometry}
\usepackage{graphicx} 

 \usepackage{array} 

\usepackage{mathtools}
\usepackage{amsfonts,hyperref,url}
\usepackage{amssymb}
\usepackage{amsthm,xcolor}
\usepackage{amsmath,bbm,bm}

\providecommand{\keywords}[1]{\textbf{\textit{Keywords: }} #1}

\newcommand{\ba}{\begin{eqnarray}}
\newcommand{\ea}{\end{eqnarray}}

    \newtheorem{theorem}{Theorem}[section]
    \newtheorem{lemma}[theorem]{Lemma}

    \newtheorem{proposition}[theorem]{Proposition}

    \theoremstyle{definition}
    \newtheorem{definition}[theorem]{Definition}

\theoremstyle{remark}
\newtheorem{remark}[theorem]{Remark}

\newtheorem{theorem*}{Theorem}
\newtheorem{lemma*}[theorem]{Lemma}
\newtheorem{corollary*}[theorem]{Corollary}
\newtheorem{proposition*}[theorem]{Proposition}
\newtheorem{problem*}[theorem]{Problem}
\newtheorem{conjecture*}[theorem]{Conjecture}

\newtheorem{examplex}{Example}
\newenvironment{example}
  {\pushQED{\qed}\examplex}
  {\popQED\endexamplex}

\newtheorem{result}[theorem]{Result}

\hypersetup{
    colorlinks=true,
    linkcolor=red,
    filecolor=blue,      
    urlcolor=cyan,
    citecolor=blue,
}

\everymath{\displaystyle}

\begin{document}
 
\title{On Completely Mixed Stochastic Games}
\author{Purba Das$^1$         \and
       T Parthasarathy$^2$ \and G Ravindran$^3$
}
\date{%
    $^1$Mathematical Institute, University of Oxford.\\%
    $^2$Chennai Mathematical Institute, India.\\
    $^3$Indian Statistical Institute, Chennai Centre, India.\\[2ex]%
    \today
}


\maketitle
\begin{abstract}
We have considered both zero-sum and non-zerosum two-person undiscounted stochastic game with finite state space and finite number of pure actions for both players. For large number of results, the transition probability of the undiscounted stochastic game is controlled by one player and all optimal strategies of the game are completely mixed. Under these assumptions, we concluded all the $\beta$-discounted zero-sum stochastic game with $\beta$ sufficiently close to $1$ and same payoff matrices are also completely mixed games. A counterexample is provided to show that the converse does not hold. Necessary conditions are provided under which the individual matrix games are completely mixed.
Under nonsingularity condition for the non-zerosum single player controlled stochastic game, we have shown that  the undiscounted game is completely mixed implies unique Nash equilibrium. Equalizer rules for completely mixed Nash equilibrium for both player controlled games are also provided.
\end{abstract}

\section{Introduction}
Stochastic game (Zero-sum) was first introduced by Shaply (\cite{shaply53}) where he showed the existence of the value of a stochastic game and existence of stationary optimal strategies for zero-sum, $\beta$-discounted stochastic games . Zero-sum undiscounted stochastic game was first introduced by Gillet (\cite{gillet}). He shows that, unlike $\beta$-discounted stochastic game, undiscounted game may not posses stationary optimal strategies. Blackwell  and  Ferguson (\cite{black1968}) studied Gillette's example and shows that though Gillette's game does not possess an stationary optimal for both the players, it has an $\epsilon$-optimal behavioral strategy for player-1 and a stationary optimal strategy for player-2.  In the same year  Blackwell  and  Ferguson (\cite{black1968}) provide an example of undiscounted stochastic game which does not possess stationary optimal strategy for one of the players. Later Filar (\cite{filr1985}) introduce the completely mixed stochastic game (previously, completely mixed matrix game was defined by Kaplansky (\cite{kapl1945}). Kaplansky's paper characterize completely mixed matrix games followed by Shapley's 
construction of `Shapley matrix' (\cite{shaply53}) builds a connection between the completely mixed matrix games and the completely mixed $\beta$-discounted stochastic games. See, for example, the recent paper \cite{sujatha2016}. Filar \cite{filr1985} extend some of the results of completely mixed matrix game to the completely mixed undiscounted stochastic game under the assumption of single player controlled transition probabilities.
\par
Nonzero-sum versions of Shapley's stochastic games \cite{shaply53} with the discounted payoff criterion were first studied by Fink (\cite{fink64}) and Takahashi (\cite{takahashi64}). The theory of nonzero-sum stochastic games with the average payoffs per unit time for the players was first introduced by Rogers (\cite{Rogers1969}) and later it was formally concidered by Sobel (\cite{Sobel1971}). They considered finite state spaces only and assumed that
the transition probability matrices induced by any stationary strategies
of the players are unichain. Parthasarathy and Raghavan (\cite{tp1981}) first considered the single player controlled stochastic game. They (\cite{tp1981}) have shown the existence of stationary strategy for non-zerosum game under the assumption of one player controlled transition. A constructive proof
of their results is given in Nowak and Raghavan (\cite{nowak93}).
\par In this paper, we have shown that under the assumption of single player controlled transition, if the zero-sum undiscounted stochastic game is completely mixed then the $\beta$-discounted stochastic games are completely mixed for all $\beta$ sufficiently close to 1. Counterexample for the converse of the previous result is also provided. Also, we have provided a necessary condition under which the individual matrix games are completely mixed. For the non-zerosum stochastic game, we have shown that under nominal assumption on the individual playoff matrices the undiscounted completely mixed stochastic game process an unique Nash equilibrium. We also have shown that for both undiscounted and $\beta$-discounted games if some Nash equilibrium (NE) of the stochastic game is completely mixed then NE of the game follows an equalizer rule. Some examples are provided to support our results. The paper ends with some open problems which are left unanswered.

\section{Definition and preliminaries}
For a finite state and finite action two players stochastic game, the game is played within two players (known as player-1 and player-2) and played in every day. Each day, the game will be in a specific state $s$ and both players will play a matrix game $R(s)$ (specified with the state $s$) and will get some reward, (which will add up to zero for the zero-sum stochastic game) which can be positive, negative or zero. The game will move to a new state in the next day and continue indefinitely.
\par Throughout the rest of the paper unless mentioned otherwise we will be assuming game means two player controlled stochastic game.
\begin{definition}[Two-person finite stochastic game]
A two person finite stochastic game can think of a $6$ tuple $G=(S,A_1,A_2,r^1,r^2,q)$. The game is played between player-1 and player-2. $S$ is the set of all states in the stochastic game. as state space is finite wlog $S=\{s_1,s_2,\cdots,s_K\}$. $A_1$ and $A_2$ be the set of all pure actions available for player-1 and player-2 respectively. That is, in state $s\in S$ player $i$ for $i=\{1,2\}$ has pure actions (finite) 
$A_i(s)=\{1,2,\cdots ,m_i\}$. $r^1$ and $r^2$ are the reward functions for respective players and $q$ is the transition 
probability function. If in state $s\in S$, player-1 and player-2 choose pure actions $i$ and $j$ respectively, then the payoff for player-1 on that specific day is $r^1(s,i,j)$ and for player-2is $r^2(s,i,j)$. With probability $q(s'|s,i,j)$ the game moves to state $s'$ on the next day.
\end{definition}
For the zero-sum stochastic game $r^2= - r^1 $ (for convenience assume $r^1 = r$). So the total gain of player-1 is preciously the total loss of player-2 on a specific day and vice-verse. Thus called zero-sum game.

The payoff matrices for player-1 and player-2 respectively in state $s\in S$ is denoted as follows.
$$R^1(s)=(r^1(s, i, j))_{m_1 \times m_2}.$$
$$R^2(s)=(r^2(s, i, j))_{m_1 \times m_2}.$$

In general, the strategy for a player can depend on the whole history up to today, but we will only be considering those strategies which do not depend on the previous history (markov strategy), that is, if the game is in state $s_0 \in S$ today then for different past history of reaching $s_0$ both the player's strategy will be exactly same. This history independent strategy is known as stationary strategy.
\begin{definition}[Stationary strategy]
Denote $P_{A_k}$, for $k\in\{1,2\}$ be the set of all probability distribution of player $k$'s action space $A_k$. Then stationary strategy of player $k$ is a function from state space $S$ to the space $P_{A_k}$.
\end{definition}
\begin{definition}[Undiscounted payoffs for stochastic game]
Denote, $r_i^{(n)}(f,g,s_0)$ for $i \in \{1,2\}$ be the expected immediate reward for player $i$ at $n^{th}$ day if the game starts in state $s_0$ and, player-1 and player-2 choose stationary strategies $f$ and $g$ respectively.
\par If player-1 and player-2 play the strategy $f$ and $g$, respectively then the undiscounted payoff player-1 and player-2 get in state $s_0\in S$ respectively is $\Phi^1(f,g)(s_0)$ and $\Phi^2(f,g)(s_0)$. Where,
 $$\Phi^1(f,g)(s_0) =\liminf_{N \uparrow \infty} \Bigg[ \frac {1}{N+1}\sum_{n=0}^{N} r_1^{(n)}( f, g,s_0)\Bigg] \; \text{and,}$$
 $$\Phi^2(f,g)(s_0) =\liminf_{N \uparrow \infty} \Bigg[ \frac {1}{N+1}\sum_{n=0}^{N} r_2^{(n)}(f, g,s_0)\Bigg].$$
\end{definition}
The `undiscounted payoff' is also known as \textbf{limiting average payoff}.
\\For the zero-sum stochastic game $\Phi^2(f,g)(s_0)=-\Phi^1(f,g)(s_0)$ (For convenience denote $\Phi^1= \Phi$).
\begin{definition}[$\beta$-discounted payoffs for general-sum stochastic game]
If player-1 and player-2 play strategies $f$ and $g$ respectively then for discount factor $\beta \in [0,1)$, the $\beta$-discounted payoff in state $s_0 \in S$ to player-1 and player-2 respectively are $I^1_\beta(f,g) (s_0)$ and $I^2_\beta(f,g) (s_0)$. Where,
$$I^1_\beta(f,g) (s_0) = \sum_{n=0}^{\infty}\beta^n  r_1^{(n)}(f, g, s_0) \; \text{and,}$$
$$I^2_\beta(f,g) (s_0) = \sum_{n=0}^{\infty}\beta^n  r_2^{(n)}(f, g, s_0).$$
Where, $r_i^{(n)}(f,g,s_0)$ for $i \in \{1,2\}$ is the expected immediate reward of 
player $i$ on $n^{th}$ day in 
state $s_0$, and player-1 and player-2 play the strategies $f$ and $g$, respectively.
\par For the zero-sum stochastic game we have $I^2_\beta(f,g) (s_0)$= $-I^1_\beta(f,g) (s_0)$ (For convenience assume $I^1_\beta= I_\beta$).
\end{definition}
\begin{definition}[Optimal strategy and value of the zero-sum stochastic game]
A pair of stationary strategies $(f^0,g^0)$ is said to be an optimal strategy in the zero-sum undiscounted stochastic game if,  for all $f\in \mathbb{P}_1$ and $g\in \mathbb{P}_2$,
$$ \Phi (f, g^0) \leqslant \Phi (f^0, g^0) \leqslant \Phi (f^0, g) \quad \text{coordinatewise.}$$
where, $\Phi(f,g) =(\Phi(f,g)(s_1), \Phi(f,g)(s_2),\cdots ,\Phi(f,g)(s_K))^{T}$. The value of the undiscounted stochastic game in state $s\in S$ is denoted as $v(s)=\Phi(f^0,g^0)(s)$.

A pair of stationary strategies $(f^0, g^0)$ is said to be an optimal stationary strategy for the zero-sum $\beta$-discounted stochastic game if, for all $f\in \mathbb{P}_{A_1}$ and for all $g\in \mathbb{P}_{A_2}$,
$$ I_\beta (f, g^0) \leqslant I_\beta (f^0, g^0) \leqslant I_\beta (f^0, g) \quad \text{coordinatewise}. $$
where, $I_{\beta}(f,g) =(I_{\beta}(f,g)(s_1), I_{\beta}(f,g)(s_2),\cdots ,I_{\beta}(f,g)(s_K))^{T}$.
The value of the $\beta$- discounted stochastic game in state $s\in S$ is denoted as $v_{\beta}(s)=I_{\beta}(f^0,g^0)(s)$.
\end{definition}
Call $f^0(g^0)$ optimal for player-1 (2) in the $\beta$-discounted game if $I_{\beta}(f^0,g)(s)\geq \min_g \max_f I_{\beta}(f,g)$ ($I_{\beta}(f,g^0)(s) \geq  \max_f \min_g I_{\beta}(f,g)(s)$) for all $g$ (for all $f$) and $s\in S$. Analogous definition holds for undiscounted stochastic game.\\
The value of a zero-sum stochastic game (both discounted and undiscounted) is always unique, whereas the optimal strategy of a stochastic game may or may not be unique.

\begin{definition}[Nash Equilibrium]
A pair of stationary strategy $(f^0,g^0)$ for player-1 and player-2 respectively is said to be Nash Equilibrium (NE) in the non-zerosum undiscounted stochastic game if,
\begin{center}
$\Phi^1(f^0,g^0) \geq \Phi^1(f,g^0)$ coordinate-wise, for all $f\in P_{A_1}$ and,\\ 
$\Phi^2(f^0,g^0) \geq \Phi^2(f^0,g)$ coordinate-wise, for all $g\in P_{A_2}.$ 
\end{center}
where, $\Phi^k(f,g) =(\Phi^k(f,g)(s_1), \Phi^k(f,g)(s_2),\cdots ,\Phi^k(f,g)(s_K))^{T}$ for $k \in \{1,2\}$ (Assuming both player-1 and player-2 wants to maximize their expected payoffs).
\\The value of the game associated with the NE $(f^0,g^0)$ in state $s\in S$ is given by $v^k_{f^0,g^0}(s)=\Phi^k(f^0,g^0)(s)$  for $k \in \{1,2\}$. For non-zerosum undiscounted stochastic games the value of the game is not unique. It changes with the NE.

A pair of strategies $(f^0,g^0)$ for player-1 and player-1 respectively is said to be a Nash Equilibrium (NE) in the $\beta$-discounted stochastic game if 
\begin{center}
$I^1_{\beta}(f^0,g^0) \geq I^1_{\beta}(f,g^0)$ coordinate-wise, for all $f\in P_{A_1}$ and,\\ 
$I^2_{\beta}(f^0,g^0) \geq I^2_{\beta}(f^0,g)$ coordinate-wise, for all $g\in P_{A_2}.$ 
\end{center}
where, $I^k_{\beta}(f,g) =(I^k_{\beta}(f,g)(s_1), I^k_{\beta}(f,g)(s_2),\cdots ,I^k_{\beta}(f,g)(s_K))^{T}$ for $k \in \{1,2\}$ (Assuming both player-1 and player-2 wants to maximize their expected payoff).
\\The value of the $\beta$-discounted game associated with the NE $(f^0,g^0)$ in state $s\in S$ is given by
$v^k_{\beta,\;f^0,g^0}(s)=
I_{\beta}^k(f^0,g^0)(s)$  for $k \in \{1,2\}$. For non-zerosum $\beta$-discounted stochastic games the value of the game is not unique. It changes with the NE.
\end{definition}

\begin{definition}[Single player controlled stochastic game (\cite{tp1981})] A stochastic game is said to be a single player
controlled stochastic game if the transition probability is controlled by only one player. For a player-2 controlled stochastic game, we have, $q(s'|s,i,j)=q(s'|s,j)$ for all $s,s'\in S$, pure action $i$ of player-1 and $j$ of player-2.
\end{definition}
Throughout the rest of the paper, unless mentioned otherwise we will be considering player-2 controlled stochastic game. An optimal strategy is said to be completely mixed if each coordinate of the strategy is strictly positive, that is each pure strategy is player with strictly positive probability.

\begin{definition}[Completely mixed stochastic game (\cite{filr1985})] A stochastic game is said to be completely mixed if every optimal strategy (Nash equilibrium) for both the players are completely mixed. That is for both player-1 and player-2 in each state $s \in S$ all the pure actions $i$ and $j$ for player-1 and player-2 respectively are played with strictly positive probabilities.
\end{definition}
The transition probability matrix $Q(f,g)$ for some stationary strategy $f$ of player-1 and $g$ 
of player-2, is defined as 
$$Q(f,g)=(q(s'|s,f,g))_{K\times K}.$$
Where, $q(s'|s,f,g)= \sum_{i=1}^{m_1} \sum_{j=1}^{m_2} f_i(s)q(s'|s,i,j)g_j(s).$
Under the assumption of player-2 controlled stochastic game, $Q(f,g)$ becomes $Q(g)=$ $(q(s'|s,g))_{K\times K}.$
\\For a stationary strategy pair (or Nash equilibrium)
$(f,g)$ the reward vector is defined as $$r^1(f,g)=(r^1(f,g,s_1),\cdots,r^1(f,g,s_K))^t$$ $$r^2(f,g)=(r^2(f,g,s_1),\cdots,r^2(f,g,s_k))^t$$ with,
$r^k(f,g,s)=f(s)^tR^k(s)g(s)=  \sum_{i=1}^{m_1} \sum_{j=1}^{m_2} f_i(s)r^k(i,j,s)g_j(s)$  for $k \in \{1,2\}$.
\par Now the discounted and undiscounted payoffs for the stationary strategy (or Nash equilibrium) $(f,g)$ for player-1 and player-2 respectively in state $s \in S$ can be written as-
\begin{center}
$I^k_{\beta}(f,g)(s)=\{[I-\beta Q(f,g)]^{-1} r^k(f,g)\}_s \quad$ and\\
$\Phi^k(f,g)(s) =\{Q^*(f,g)r^k(f,g)\}_s $
\end{center}
Where, $k \in \{1,2\}$, $Q^0(f,g) = I$ and the markov matrix $Q^*(f,g)$ is as follows.
$$Q^*(f,g)= \lim_{n \rightarrow \infty} \frac{1}{N+1}\sum_{n=0} ^N Q^n(f,g).$$ And the notation $\{.\}_s$ represents the $s^{th}$ coordinate of the respective vector.
\begin{definition}[Uniform Discount Optimal]
A strategy $g^0$ for player-2 is said to be an uniform discount optimal if $g^0$ is optimal for player-2 in the undiscounted stochastic game $\Gamma$ as well as in the $\beta$-discounted game $\Gamma_{\beta}$ for all $\beta$ close to $1$. Similarly we can extend the definition for player-1 also.
\end{definition}

\begin{definition}[Uniform Discount Equilibrium]
A Nash equilibrium pair $(f^0,g^0)$ for player-1 and player-2 respectively is said to be an uniform discount equilibrium if $(f^0,g^0)$ is a Nash equilibrium in the undiscounted stochastic game $\Gamma$ and $(f^{\beta},g^0)$ is a Nash equilibrium in the $\beta$-discounted game $\Gamma_{\beta}$ for all $\beta$ close to $1$.
Similarly we can define for player-1 also.
\end{definition}

\begin{definition}[Auxiliary Game]
For a $\beta$-discounted zero-sum stochastic game the auxiliary game in state $s \in S$ is defined by the matrix $R_\beta(s)$. The $(i,j)^{th}$ entry of the matrix $R_\beta(s)$ is given by
$$r(i,j,s) + \beta \sum_{s' \in S } v_\beta(s') q(s'|s,i,j)$$
where, $v_\beta(s')$ is the value of the $\beta$-discounted stochastic game at state $s'\in S$. $R_{\beta}(s)$ is called the auxiliary game at state $s$ (or starting at state s). The matrix $R_{\beta}(s)$ is known as, \textbf{Shapley Matrix} (\cite{shaply53}).
\end{definition}
The following results will be used to proof our results.
\begin{result}[Theorem 1 page 475, \cite{kapl1945}]\label{result1}
Consider a two person zero sum matrix game
with payoff matrix $M \in \mathbb{R}^{m\times n}$. Suppose player-2 has a completely mixed 
optimal strategy $y'$ then for any optimal strategy strategy $x'$ for player-1, we have $\sum_{i=1}^{m}m_{ij}x_i' \equiv v $ for all $j\in \{1,2, \cdots ,n\}$, where $v$ is the value of the matrix game and $m_{ij}$ is the $(i,j)^{th}$ entry of the matrix $M$.
\end{result} 

\begin{result}[\cite{tp1981}]\label{result2}
For player-2 controlled stochastic game (both zero-sum and non-zerosum) if $(f^{\beta},g^{\beta})$ is a pair of Nash equilibrium in the $\beta$-discounted stochastic game and if $f^{\beta}\rightarrow f^0$ component wise then, for $k=\{1,2\}$,
$$\lim_{\beta \uparrow 1}(1-\beta)I^k_{\beta}(f^{\beta},g) \equiv \Phi^k(f^0,g)$$
where, $g$ is a stationary strategy for player-2. \cite{tp1981} also provides,
\begin{center}
$v^k(s)=\lim_{\beta \uparrow 1}(1-\beta)v_{\beta}(s)$ 
\quad for all $s \in S$ and $k\in \{1,2\}$.
\end{center}
\end{result}
\begin{result}[Shapley Matrix \cite{shaply53}]\label{result3}
A two-player zero-sum $\beta$-discounted stochastic game $\Gamma_\beta = (S, A_1, A_2$ $,r, q, \beta)$ has an optimal value vector $v_\beta$ which is obtained by $$v_\beta(s) = val(R_\beta(s)),$$
where $R_\beta(s)$ is the shapley matrix in state s.
\\ For each state $s \in S$, $(f^\beta(s), g^\beta(s))$ is a pair of optimal strategies of the matrix game $R_\beta(s)$ , if and only if $(f^\beta, g^\beta)$ is a pair of optimal strategies for the $\beta$-discounted stochastic game $\Gamma_{\beta}$, where\\ $f^\beta = (f^\beta(s_1),f^\beta(s_2), \cdots , f^\beta(s_K))$ and $g^\beta = (g^\beta(s_1), g^\beta(s_2), \cdots , g^\beta(s_K))$ and $K$ is the number of states.
\end{result}

\section{Zero-sum Games}\label{zerosum}
Unless mentioned otherwise for Section \ref{zerosum}, player-1 is the maximizing player and player-2 is the minimizing player. The following Lemma is required to proof the main theorem.
\begin{lemma}\label{lemma1}
Assume player-2 controlled transition. Suppose, there exists $\beta_0 \in [0,1)$ and a completely mixed stationary strategy $g^0$ such that $g^0$ is optimal for player-2 (Minimizer) in \textbf{every} $\beta$-discounted stochastic game for all $\beta > \beta_0$.
Let, $\beta_n \in [\beta_0,1)$ be such that $\beta_n \uparrow 1$. Let $\{f_n\}$ be optimal for player-1 (Maximizer) for $\beta_n$-discounted stochastic game. Suppose $f_n \rightarrow f_0$ coordinate-wise, that is $f_n(s) \rightarrow f_0(s)$ for each state $s\in S$ then $f_0$ is optimal for player-1 in the undiscounted stochastic game.
\end{lemma}
\begin{proof}
Under one player controlled transition probability, an undiscounted stochastic game has value restricted to stationary strategy (\cite{tp1981}).\\
Suppose $f_n$ is optimal for player-1 in the $\beta_n$-discounted stochastic game. From the assumption we have a completely mixed strategy $g_0$ for player-2 in the $\beta_n$-discounted stochastic game from Result \ref{result1} and Result \ref{result3}:
$$I_{\beta_n}(f_n,g) \equiv v_{\beta_n} \quad \text{and} \quad I_{\beta_n}(f_n,g)=(I_{\beta_n}(f_n,g)(s_1),\cdots,I_{\beta_n}(f_n,g)(s_K))^t,$$
where, $g$ is any stationary strategy of player-2, and
$$v_{\beta_n}=(v_{\beta_n}(s_1),\cdots,v_{\beta_n}(s_K))^t.$$
Therefore, we have 
$$[I-\beta_n Q(g)]^{-1}r(f_n,g)\equiv v_{\beta_n}.$$
Now $[I-\beta_n Q(g)]^{-1} =\sum_{k=0}^\infty\beta_n^k Q^k(g)$  is a non-negative matrix. Hence the expression of $r(f_n,g)$ is as follows.
$$ r(f_n,g) \equiv [I-\beta_n Q(g)] v_{\beta_n}. $$
Also, we have $f_n \rightarrow f_0$ point-wise and the reward function $r(.,.)$ is a continuous function on player-1's strategy. Hence, for any given $\epsilon > 0 $ there exists $N_0\in \mathbb{N}$ such that for all $n\geq N_0$:
$$[I-\beta_n Q(g)] v_{\beta_n} - \epsilon e \leq r(f_0,g) \leq [I-\beta_n Q(g)] v_{\beta_n}+ \epsilon e$$
coordinate-wise, where $e$ is a suitable length column vector with all entry as $1$. Therefore we have,
$$v_{\beta_n} - \epsilon [I-\beta_n Q(g)]^{-1} e \leq [I-\beta_n Q(g)]^{-1} r(f_0,g) \leq  v_{\beta_n}+ \epsilon [I-\beta_n Q(g)]^{-1} e.$$
As $(1-\beta_n)$ is always non-negative for all $\beta_n \in (0,1]$, we have.
$$(1-\beta_n)v_{\beta_n} - (1-\beta_n) \epsilon [I-\beta_n Q(g)]^{-1} e \leq (1-\beta_n) [I-\beta_n Q(g)]^{-1} r(f_0,g) $$$$\leq (1-\beta_n) v_{\beta_n}+ (1-\beta_n) \epsilon [I-\beta_n Q(g)]^{-1} e.$$
We have, $[I-\beta_n Q(g)]^{-1}e= [\sum_{k=0}^{\infty} \beta_n^k Q^k(g)]e = \sum_{k=0}^{\infty} \beta_n^k [Q^k(g)e] = \sum_{k=0}^{\infty} \beta_n^k e$ as $Q^k$ is a stochastic matrix for each $k$. Therefore the above inequality reduces to the following inequality.
$$(1-\beta_n)v_{\beta_n} -  \epsilon  e \leq (1-\beta_n) [I-\beta_n Q(g)]^{-1} r(f_0,g) \leq (1-\beta_n) v_{\beta_n}+ \epsilon e.$$
Now if we let $\beta_n \uparrow 1$, Using Result \ref{result2} we have the above inequality as follows.
$$v-\epsilon e \leq \Phi(f_0,g) \leq v+\epsilon e, $$ 
where $v=(v(s_1),\cdots,v(s_K))^t$ is the value of the undiscounted stochastic game. The argument is true for any $\epsilon > 0 $. Hence $f_0$ constructed above is optimal strategy for player-1 in the undiscounted stochastic game $\Gamma$. 
 \end{proof}
If an undiscounted single player stochastic game is completely mixed then we can conclude the $\beta$-discounted stochastic games are completely mixed for $\beta$ sufficient close to $1$.

\begin{theorem}\label{main.theorem}
Consider a finite, undiscounted, zero-sum, single player controlled completely mixed stochastic game $\Gamma$. Then there exists $\beta_0 \in [0,1)$ such that for all $ \beta > \beta_0 $, the $\beta$-discounted stochastic games $\Gamma_\beta$ obtained from the same payoff matrices are completely mixed. 
\end{theorem}
\begin{proof}
Let us assume by contradiction, that the undiscounted stochastic game is completely mixed but there does not exist any $0 \leq \beta_0<1$ such that the $\beta$-discounted game are completely mixed for all $\beta>\beta_0$.
\par Then given any $\beta\in [0,1)$ we can find  a $\beta_1$ $\in (\beta_0,1)$ such that there will exists a non-completely mixed strategy $f_1$ which is optimal for the $\beta_1$-discounted stochastic game.
\par Similarly, we can find a $\beta_2$ $\in (\beta_1,1)$ such that there will exist a non-completely mixed strategy $f_2$ which is optimal for the $\beta_2$-discounted stochastic game, as from the assumption above.
Hence we will obtain a sequence $\beta_n \uparrow 1$ such that for all $\beta_n$-discounted game we have non-completely mixed strategy $f_n$. But we have finitely many states and finitely many pure actions in the stochastic game.
Hence there exists a state $\bar{s}\in S$ and a pure action $i$ for player-1 such that the $i^{th}$ coordinate of $f_n(\bar{s})$ is zero for infinitely many $f_n$. That is, we can find a sub-sequence $f_{n_k}$ of $f_n$ for which a fixed state $\bar{s}\in S$ and a fixed pure action $i$ can be found, such that the $i^{th}$ pure strategy in the state $\bar{s}^{th}$ is always played with probability zero in the optimal strategy. Hence applying the Lemma \ref{lemma1}, limit of this sub-sequence $f_{n_k}$ of $f_n$ (denote as $f_0$) is optimal strategy for player-1 in the undiscounted game. However $f^0$ is not completely mixed, yet it is optimal in the undiscounted stochastic game. This is a contradiction. So there exists $\beta_0 \in [0,1)$ such that forall $ \beta > \beta_0 $ the $\beta-$discounted stochastic game $\Gamma_\beta$ obtained from the same payoff matrices are completely mixed. 
\end{proof}
The converse of the above theorem is not true. Let us consider an example of finite player-2 controlled zero-sum stochastic game where the $\beta$-discounted stochastic game is completely mixed for all $\beta \in [0,1)$ but the undiscounted stochastic game is \textbf{not} completely mixed.
\begin{example}
$$R(s_1) =
\begin{bmatrix}
    0/(0,1)      & 2/(0,1)   \\
    3/(0,1)       & 1/(0,1) 
\end{bmatrix}      ,       R(s_2) =
\begin{bmatrix}
    2/(0,1)      & 0/(0,1)   \\
    0/(0,1)       & 2/(0,1) 
\end{bmatrix}$$
\vspace{2mm}
\textit{Note:} $s_2$ is an absorbing state. 

Consider the $\beta$-discounted game $\Gamma_\beta$ with the mentioned states and actions.
\\ The Shapley matrix (\cite{shaply53}) $R_\beta(.)$ is given by,
$$R_\beta(s_1) =\begin{bmatrix}
    0+\frac{\beta}{1-\beta}      & 2+\frac{\beta}{1-\beta}   \\
    3+\frac{\beta}{1-\beta}       & 1+\frac{\beta}{1-\beta} 
\end{bmatrix}, \quad
R_\beta(s_2) =\begin{bmatrix}
    2+\frac{\beta}{1-\beta}      & 0+\frac{\beta}{1-\beta}   \\
    0+\frac{\beta}{1-\beta}       & 2+\frac{\beta}{1-\beta} 
\end{bmatrix} $$

\item[] Clearly, the matrix $R_\beta(s_1)$ is completely mixed for all $\beta \in [0,1)$.
\item[]Furthermore, $R_\beta(s_2)$ is also completely mixed. Hence the game $\Gamma_\beta$ is completely mixed for all $\beta \in [0,1)$.
\item But if we consider the undiscounted stochastic game $\Gamma$. Consider the strategy $f$ with $f(s_1)=(1,0) ,f(s_2)=(\frac{1}{2},\frac{1}{2})$. $f$ is a stationary optimal strategy for player-1 in the undiscounted stochastic game $\Gamma$ but not completely mixed. Hence the undiscounted game $\Gamma$ is not completely mixed. 
 \end{example}

\begin{lemma}\label{lemma2}
Let, $A=(a_{ij})$ be a symmetric matrix of order $n$ with $a_{ij}>0$ for every $i$ and $j$. Let $b^t=(b_1,b_2,\cdots,b_n)$ be a non-negative vector. Let, $C=(c_{ij})$ where, $c_{ij}=a_{ij}+b_j$. Then $C$ completely mixed matrix game implies $A$ is also a completely mixed matrix game. 
\end{lemma}
\begin{proof}
Suppose, $C$ is completely mixed. Since the value $v$ of $C$ is positive, from \cite{kapl1945} $\det(C) \neq 0$. Also from Lemma 4.1, page 381; \cite{tp1981} we have $\det(A) \neq 0$.
\par Let $y$ be a completely mixed optimal strategy for $C$. Then $Cy=ve$, where $e$ is a vector with all coordinates equals to $1$ and $v$ denotes the value of the matrix game $C$. It follows that $Ay=(v-b^ty)e=\delta e$, where $\delta = v-b^ty$. Since $A$ is symmetric, it follows that $\delta$ is the value of the matrix game $A$ and $(y,y)$ is an optimal strategy of $A$.
\par To complete the proof we will show that $y$ is the only optimal strategy for both players in the matrix game $A$. Let, $z$ be any other optimal strategy for player-2 in $A$. Then $Az=\delta e$. Since $y$ is a completely mixed optimal for both players (using Result \ref{result1}) we also have $Ay=\delta e$. Thus $A(y-z)$ equals to zero vector. Since $A$ is non-singular, $y=z$. In other words, every optimal strategy for player-2 in $A$ coincides with $y$. Similar argument holds for player-1 in $A$. This completes the proof. 
\end{proof}
A similar type of proof is also provided in \cite{sujatha2016}. We now give a simple counterexample to show that if $A$ is completely mixed matrix game, then $C$ need not be completely mixed matrix game.
\begin{example}
Let, $A=\begin{bmatrix}
    1 &2\\ 2& 1
\end{bmatrix}$ and let $b=(b_1,b_2)^t = (1,2)^t$, then $C=\begin{bmatrix}
    2 & 4\\ 3 & 3\end{bmatrix}$.
$A$ is a completely mixed matrix game but $C $ is not a completely mixed matrix game. So the converse of Lemma \ref{lemma2} is not true.
\end{example}
\textit{Remark:} Lemma \ref{lemma2} also holds under both the following circumstance.  
\\1. a pair of completely mixed optimal strategies exist for the two players in $C$ instead of $C$ to be a completely mixed matrix game.
\\2. $A$ to be a symmetric, nonnegative and irreducible instead of $A$ to be a strictly positive matrix.
\par Undiscounted completely mixed stochastic games not only proceed completely mixed $\beta$-discounted stocha- stic games for large enough $\beta$ (see Theorem \ref{main.theorem}) but also process completely mixed matrix games under some symmetry assumption. 
\begin{theorem}\label{theorem2}
Consider a finite, undiscounted, zero-sum, single player (player-2) controlled completely mixed stochastic game $\Gamma$. Assume all the individual payoff matrices $R(s)$ are symmetric then the individual matrix games $R(s)$ is completely mixed for all $s \in S$.
\end{theorem}
\begin{proof}
Assume without loss of generality $R(s)$ is positive, that is every entry in $R(s)$ is positive. Take any $\beta > \beta_0$ where, $\beta_0$ is from Theorem \ref{main.theorem}. Then the matrix game $R_{\beta}(s)$ (also known as Shapley matrix) is completely mixed for all $s\in S$, where $R_{\beta}(s)=R(s)+b(s)$, with $b_j(s)=\beta \sum_{s' \in S}v_{\beta}(s')q(s'|s,j)$. Now the result is immediate from Theorem \ref{main.theorem} and Lemma \ref{lemma2}. 
\end{proof}

Example of completely mixed undiscounted stochastic game are far from obvious. The following is a finite single player controlled completely mixed stochastic game.
\begin{example}\label{example3}
$$R(s_1) =
\left[\begin{matrix}
    2/(1,0)      & 0/(0,1)   \\[10pt]
    0/(1,0)       & 2/(0,1) 
\end{matrix}\right] \quad     , \quad      R(s_2) =
\left[\begin{matrix}
    3/(1,0)      & -1/(0,1)  \\[10pt]
    -1/(1,0)     & 3/(0,1)
\end{matrix}\right]
$$
This is a player-2 controlled stochastic game with 2 states $s_1$ and $s_2$. In both state $s_1$ and $s_2$ if player-2 chooses action 1 (column 1), then the game moves to state $s_1$ in the next day and if player-2 chooses action 2 (column 2), then the game moves to state $s_2$ in the next day.

The unique optimal strategy for player-1 is $f= \{(\frac{1}{2},\frac{1}{2}), (\frac{1}{2},\frac{1}{2})\}$ in the above-mentioned undiscounted game. That is, choosing row 1 and row 2 is state $s_1$ with probability $\frac{1}{2}$ and $\frac{1}{2}$ respectively and choosing row 1 and row 2 is state $s_2$ with probability $\frac{1}{2}$ and $\frac{1}{2}$ respectively. The player-2's unique optimal strategy is $g=\{(\frac{1}{2},\frac{1}{2}),(\frac{1}{2},\frac{1}{2})\}$. So $(f,g)$ is the unique optimal strategy for the undiscounted game mentioned before.
\end{example}
In Example \ref{example3}, we can see that for all $\beta \in [0,1)$ the $\beta$-discounted stochastic game $\Gamma_{\beta}$ are completely mixed. As both $s_1$ and $s_2$ are symmetric matrix Theorem \ref{theorem2} says individual matrix games are completely mixed, which is easy to see in the above example. 

Similar to Theorem \ref{main.theorem}, if we know the value of an undiscounted stochastic game is non-zero then we can conclude the value for the $\beta$-discounted games are also non-zero. 
\begin{theorem}\label{theorem3}
Assume a player-2 controlled transition. Let, $v(s)$ and $v_\beta(s)$ be the value of an undiscounted and $\beta$-discounted stochastic game in state $s\in S$ respectively. If $v(s) \neq 0 $ for some $s \in S$ then there exists $\beta_0\in [0,1)$ such that for all $\beta > \beta_0$, the discounted value $v_\beta(s) \neq 0 $. 
\end{theorem}
\begin{proof}
Since $v(s) \neq 0$ from \cite{tp1981} there exist $(f^0,g^0)$ optimal strategy in $\Gamma$ such that $\lim_{\beta \uparrow 1}f^{\beta}(s) = f^0(s)$. And $(f^\beta,g^0)$ is optimal stationary strategy in $\Gamma_{\beta}$ for all $\beta > \beta_0$ for some $\beta_0\in [0,1)$.
Now without loss of generality we assume the reward $r(i,j,s)>0$ for all $i\in \mathbb{P}_1$, $j\in \mathbb{P}_2$ and $s\in S$. Hence, $R(s)>0$ which implies $v(s)>0$ for all $s\in S$. Now,   
$$0 < v(s) = \Phi(f^0,g^0)(s) \leqslant \Phi(f^0,g)(s)$$ for all $g \in \mathbb{P}_{A_2}$ (Since we assume player  is maximizer and player-2 is minimizer).
Therefore,
 $$ 0< \Phi(f^0,g^0)= \lim_{\beta \uparrow 1} (1-\beta) I_\beta(f^0,g^0)(s).$$ Implies,
$ \lim_{\beta \uparrow 1} (1-\beta) I_\beta((\lim_{\beta \uparrow 1}f^\beta),g^0)(s)  > 0 $. This follows,
 $\lim_{\beta \uparrow 1} (1-\beta) I_\beta(f^\beta,g^0)(s)> 0$ (as for player-2 controlled stochastic game $I_{\beta}$ is continious and linear in player-1's strategy). Then, there exists $\beta_0\in [0,1)$ such that for all $\beta >\beta_0$ we have $ I_\beta(f^\beta,g^0)(s) > 0$.
Hence the Lemma follows.
 \end{proof}
\textit{Note:} The above theorem is true for non-zerosum stochastic game also. For non-zerosum stochastic game the proof will be exactly similar to the above proof using results of \cite{tp1981}. 
\par \textit{Note:} The converse of the Theorem \ref{theorem3} is not true. The following example shows that.
\begin{example}\label{example4}
Consider the following player-2 controlled stochastic game: 
$$R(s_1) =
\begin{bmatrix}
    4/(0,1,0)      & 2/(0,0,1)   \\
    3/(0,1,0)       & 1/(0,0,1) 
\end{bmatrix}      , \;      R(s_2) =
\begin{bmatrix}
    2/(0,1,0)      & 0/(0,1,0)  \\
    0/(0,1,0)     & 2/(0,1,0) 
\end{bmatrix}  ,$$ $$       R(s_3) =
\begin{bmatrix}
    1/(0,0,1)      & -1/(0,0,1)  \\
    -1/((0,0,1)     & 1/(0,0,1) 
\end{bmatrix}$$
\end{example}
States $s_2$ and $s_3$ are absorbing states. $R_{\beta}(s_1)$ is the auxiliary game (also known as Shapley matrix) in state 1 (denoted as $s_1$). For all $\beta \in [0,1)$ $val(R_{\beta}(s_1)) = 2$, where $val(.)$ denotes the value of the corresponding matrix game. But the undiscounted value in state 1 is given by-
$ v(s_1) = \lim_{\beta \uparrow 1}(1-\beta) v_{\beta}(s_1) =0.$

\begin{lemma}\label{lemma3}
Let $C$ be a completely mixed matrix game. $c_{ij}=a_{ij}+b_j$ where $a_{ij}, b_{j}>0 \; \forall \; i, j.$ If $val(A) = val(A^t)$ then $val(A)=val(C)-b^ty_0$ where, $y_0$ is optimal strategy for player-2 in the matrix game $C$. Also, we have player-1 has an unique optimal in the matrix game $A$.
\end{lemma}
\begin{proof}
We skip the proof as it is straight forward. 
\end{proof}
Under the symmetric assumption, the undiscounted value satisfies linear equations. 
\begin{theorem}
Let $\Gamma$ be player-2 controlled completely mixed undiscounted stochastic game. Let for all $s \in S$, the individual payoff matrices $R(S)$ proceeds $val(R(s))=val(R(s)^t$ where, $val(.)$ denotes the value of the corresponding matrix game. Then, for $g^0$-the unique optimal for player-2 in undiscounted game $\Gamma$ the following equality holds.
$$\text{ for all } s\in S, \qquad v(s)=\sum_{s'\in S} v(s')q(s'|s,g^0).$$
\end{theorem}
\begin{proof}
The Shapley matrix in state s, for the $\beta$-discounted stochastic game $\Gamma_{\beta}$ with same payoff matrix $R(s)$ as of $\Gamma$ is denoted as $R_{\beta}(s)$.
$$R_{\beta}(s) = \left[r(i,j,s)+\beta\sum_{s'\in S} v_{\beta}(s')q(s'|s,j)\right]_{ij}$$
Now from Theorem \ref{main.theorem} we know that the $\beta$-discounted stochastic game $\Gamma_{\beta}$ with same payoff matrix is completely mixed for all $\beta > \beta_0$. This implies $\forall s\in S$ and $\forall \beta$ sufficiently close to $1$ (say $\beta > \beta_0$) the Shapley matrix  $R_{\beta}(s)$ is completely mixed.
\\ Assume the unique optimal (\cite{filr1985}) in the undiscounted game to be $(f^0,g^0)$. We also have $g^0$ is an uniform discount optimal for player-2 in the undiscounted game $\Gamma$.
\\ Without loss of generality we will assume the matrix $R(s)>0 \forall s\in S$. Take $A=R(s)$ and $C=R_{\beta}(s)$ in Lemma \ref{lemma3} gives,
$$val(R(s))=val(R_{\beta}(s))-\sum_{j} \beta\sum_{s'\in S} v_{\beta}(s')q(s'|s,j)g^0(s)_j,$$
where $g^0(s)_j$ is the $j^{th}$ coordinate of $g^0(s)$. So we have the following.
$$ val(R(s))=v_{\beta}(s)- \beta\sum_{s'\in S} v_{\beta}(s')\sum_{j}q(s'|s,j)g^0(s)_j.$$
Which then follows:
$$val(R(s))=v_{\beta}(s)- \beta\sum_{s'\in S} v_{\beta}(s')q(s'|s,g^0).$$
Multiplying by $(1-\beta)$ in both the sides we obtain:
$$ (1-\beta)val(R(s))=(1-\beta)v_{\beta}(s)- \beta(1-\beta)\sum_{s'\in S} v_{\beta}(s')q(s'|s,g^0).$$ 
The above equality is true for all $\beta \in [0,1)$. Taking limit we have the following.
$$ \lim_{\beta \uparrow 1} (1-\beta)val(R(s)) = \lim_{\beta \uparrow 1}(1-\beta)v_{\beta}(s)- \lim_{\beta \uparrow 1}\beta\sum_{s'\in S} (1-\beta) v_{\beta}(s') q(s'|s,g^0)$$ 

$$\implies 0 = v(s)- \sum_{s'\in S} v(s') q(s'|s,g^0)$$

\end{proof}
\begin{remark}\label{remark1}
If $\Gamma$ is an undiscounted non-zerosum single player (player-2) controlled completely mixed stochastic game then all the strictly positive strategy for player-2 partitions the states $S$ of undiscounted game $\Gamma$ into $k$ sets of ergodic chains
$C_1,C_2,\cdots,C_k$ and a set $H$ of transient states. Also the stochastic game $\Gamma$ can be divided into $k$ subgames $\Gamma_1,\Gamma_2,\cdots, \Gamma_k,$ with value that is independent of states in each of the game $\Gamma_s$ corresponding to states $C_s$.
\\For the zero-sum game in \cite{rag1984} it is shown that if $s\in H $ then player-1 and player-2 both has exactly one action in the state $s\in S$.
\end{remark}
Using Remark \ref{remark1} and the fact that for player-2 controlled stochastic game the value is a linear function of player-1's strategy we can show that for an undiscounted stochastic game every completely mixed optimal strategy proceeds an equalizer rule in term of undiscounted value.
\begin{theorem}\label{theorem5}
Let $\Gamma$ be a zero-sum undiscounted stochastic games where player-2 controls the transition probability. If $(f^0,g^0)$ is a completely mixed optimal strategy in the game $\Gamma$ then the game follows the equalizer rule.
\[\text{ for all } f\in \mathbb{P}_{1} \text{ and } g\in \mathbb{P}_{2}\qquad \Phi(f^0, g^0)=\Phi(f^0, g)=\Phi(f, g^0).\]
\end{theorem}
\begin{proof}
Filar (\cite{filr1985}) shows Theorem \ref{theorem5} when the strategies $f$ and $g$ are restricted to only the pure strategies of the respective players under the completely mixed assumption of the game.

Since the stochastic game $\Gamma$ is completely mixed from Remark \ref{remark1} we can conclude that it is sufficient to look at the stochastic games $\Gamma_c$ separately. Now for all the subgame  $\Gamma_c$ with states restricted to $C_c$ are completely mixed. Also, if $(f^0,g^0)$ is optimal for the stochastic game $\Gamma$ then $(f^0,g^0)$ restricted to state $C_c$ (denoted as $(f^0,g^0)_c)$ is optimal strategy in the stochastic game $\Gamma_c$.     

If we fixed player-1's (unique) optimal strategy $f^0_c$ then there exists vector $\gamma^0$ along with value $v_c\textbf{1}$ of the game $\Gamma_c$, which satisfy the following equality (\cite{Hordijk1979}):
$$v_c+ \gamma^0(s)= r_c(f^0,j,s)+  \sum_{s' \in S} q(s'|s,j)\gamma^0(s')$$
for all pure strategy $j$ for player-2 and for all $s\in C_c$. This implies 
$$ v_cg_j(s)+ \gamma^0(s)g_j(s)= r_c(f^0,j,s)g_j(s)+  \sum_{s' \in S} q(s'|s,j)g_j(s)\gamma^0(s')$$
for all player-2's pure strategy $j$, for all $s\in C_c$ and $g\in \mathbb{P}_{A_2}$. So,
$$ \sum_{j} v_cg_j(s)+ \sum_{j}\gamma^0(s)g_j(s)=\sum_{j} r_c(f^0,j,s)g_j(s)+ \sum_{j} \sum_{s' \in S} q(s'|s,j)g_j(s)\gamma^0(s')$$
for all $s\in C_c$. So we obtain for all $s\in C_c$:
$$ v_c+ \gamma^0(s)= r_c(f^0,g,s)+  \sum_{s' \in S} q(s'|s,g)\gamma^0(s').$$
Writing the above equality in a vector notation we obtain.
$$ v_c\textbf{1}+ \gamma^0= r_c(f^0,g)+  Q(g)\gamma^0$$
for all $g\in \mathbb{P}_{A_2}$. Multiplying the Markov matrix $Q^*(g)$ in both the sides of the above equality we get:
$$v_c\textbf{1}+ Q^*(g)\gamma^0= \Phi_c(f^0,g)+  Q^*(g)\gamma^0.$$
Which indeed follows,
$$ \text{ for all } g\in \mathbb{P}_{A_2} \qquad \Phi_c(f^0,g) =\Phi_c(f^0,g^0).$$
\\ The above argument is true for all the individual completely mixed games $\Gamma_1,\Gamma_2,\cdots\Gamma_k$. Now if $s\in H$ ($H$ is the set of all transient states) then let $p_g(s,s')$ be the probability that $s'$ is the first state reached outside $H$. So for all $s\in H$,
$$\Phi(f^0,g)(s)= \sum_{s\in H^c} p_g(s,s') \Phi(f^0,g)(s').$$
We already have $\Phi_c(f^0,g)(s) =\Phi_c(f^0,g^0)(s)$ for $s\in H^c$. Hence the claim is proved.
\\The other side of the equality directly follows from the Proposition 3.3; \cite{filr1985} and the fact that $\Phi$ is continuous on player-1's strategy.
\end{proof}
Let $\Gamma$ be a player-2 controlled zero sum undiscounted stochastic game with individual payoff matrix $R(s)$ for all $s \in S$. Let $\Gamma_{\beta}$ be the corresponding $\beta$-discounted stochastic game with same payoff matrix $R(s)$.
\\ We say a stationary strategy $g^0$ is \textbf{uniform discount optimal} for player-2 in the game $\Gamma$ if there exists $\beta_0 \in [0,1)$ such that $g^0$ is optimal for player-2 in the corresponding $\beta$-discounted game $\Gamma_{\beta}$ for all $\beta > \beta_0$.

\begin{remark}[\cite{tp1981}]\label{remark2}
For $\beta$-discounted stochastic game for each $s\in S$ there exists a nonsingular sub-matrix $\dot{R}_{\beta}(s)$ of Shapley matrix $R_{\beta}(s)$ 
such that if we define for all $\beta > \beta_0$
\begin{center}
$\dot{f}^{\beta}(s)= v_{\beta}(s)\textbf{e}[\dot{R}_{\beta}(s)]^{-1}$
and,\\
$\dot{g}^{\beta}(s)= v_{\beta}(s)[\dot{R}_{\beta}(s)]^{-1}\textbf{e},$
    
\end{center}
where, $e$ is the column vector of all $1$'s. Then the set of pair $(f^{\beta},g^{\beta})$ obtained from $(\dot{f}^{\beta}(s),\dot{g}^{\beta}(s))$ by adding zero in the place corresponding to the rows/ columns of $\dot{R}(s)$ which are not in $R(s)$ from an optimal stationary strategy pair $(f^{\beta},g^{\beta})$ in $\Gamma_{\beta}$ for all $\beta > \beta_0$.
\\ Furthermore for all $s\in S$ it is also shown that, $\dot{R}(s)$ is non-singular and $\dot{g}^{\beta_1}(s)=\dot{g}^{\beta_2}(s)$ for all $\beta_1, \beta_2 > \beta_0$. Denote, $\dot{g}^{\beta_1}(s) =\dot{g}^{0}(s) \quad \forall \beta> \beta_0$. Then $g^0=(g^0(1),g^0(2),\cdots,g^0(S))^t$ where $g^0(s)$ is obtained by completing $\dot{g}(s)$ with $0's$ is an uniformly discount optimal for player-2 in $\Gamma$. Also it turns out that $(f^0,g^0)$ where $f^0(s)= \lim_{\beta \uparrow 1} f^{\beta}(s)$ is an optimal strategy pair in the undiscounted game $\Gamma$.
 
\end{remark}
\begin{definition} (\cite{rag1984})
A game (matrix,bimatrix,stochastic) is called `CM-I' if player-1 only has completely mixed optimal strategies.
Similarly, we can define `CM-II'.
\end{definition}
For undiscounted stochastic game this definition was first introduced by Filar(\cite{filr1981}).
In matrix game, bimatrix game and discounted stochastic game if both player has same number of pure strategy than CM-I implies Completely mixed game and hence CM-II. Similarly CM-II also implies completely mixed game and hence CM-I. But for undiscounted stochastic game this is far form reality. 

%
\begin{lemma}\label{lemma4}
If an undiscounted player-2 controlled stochastic game is CM-I then player-2 has a completely mixed optimal in $\Gamma$ and there exists an $\beta_0 \in [0,1)$ such that for all $\beta > \beta_0$ the $\beta$-discounted stochastic game $\Gamma_{\beta}$ is completely mixed.
\end{lemma}
\begin{proof}
The proof goes exactly in the same way of Theorem \ref{main.theorem}. And since for $\beta$ closed to $1$ $\Gamma_\beta$ is completely mixed, the discount optimal for player-2 is also completely mixed. Hence that optimal is a completely mixed optimal for player-2 in the undiscounted stochastic game.
 \end{proof}
Lemma \ref{lemma4} only says the $\beta$-discounted stochastic games are completely mixed. It is an unsolved question if the undiscounted stochastic game is completely mixed or not.



\section{Non-zerosum game}\label{nonzero}


Unless mentioned otherwise throughout Section \ref{nonzero} both player-1 and player-2 are maximizer.
For the non-zerosum discounted stochastic game we have the following result which is needed to prove the main result.
\begin{theorem}(\cite{fillerbook1996})\label{theorem6}
The following assertions are equivalent:
\\ (i) $(f^0, g^0)$ is an equilibrium point (EP) in the discounted stochastic game with
equilibrium payoffs\\ $(v_{\beta}^1(f^0, g^0), v_{\beta}^2(f^0, g^0))$.
\\(ii) For each $s \in S$, the pair $(f^0(s),g^0(s))$ constitutes an equilibrium point in the static bimatrix game $(B^1_{\beta}(s), B^2_{\beta}(s))$ with equilibrium payoffs
$(v_{\beta}^1(f^0, g^0)(s), v_{\beta}^2(f^0, g^0)(s))$, where
  		
$$B^1_{\beta}(s) = \Bigg[(1-\beta)r^1(i,j,s)+\beta\sum_{s'\in S} v^1_{\beta}(s')q(s'|s,i,j)\bigg]_{ij}$$
  		
$$B^2_{\beta}(s) = \Bigg[(1-\beta)r^2(i,j,s)+\beta\sum_{s'\in S} v^2_{\beta}(s')q(s'|s,i,j)\bigg]_{ij}$$
  		
\end{theorem}
For bimatrix game we have an equalizer rule for the payoff under certain compeltely mixed assumption.
\begin{lemma}\label{lemma5}
Let $(A,B)$ be a bimatrix game. If $(x^0,y^0)$ be a EP in the bimatrix game with $y^0$ completely mixed then $r^2(x^0,y)=r^2(x^0,y^0)$ for all player-2's stationary strategy $y$.
\end{lemma}
\begin{proof}
Since $(x^0,y^0)$ is an equilibrium point for all stationary strategy $y$ of player-2 we have $x^{0^{t}}By^0\geq x^{0^{t}}By$. Taking $y$ as $(1,0,\cdots,0) ,(0,1,\cdots,0) ,\cdots,(0,0,\cdots,1)$ 
respectively we get ${x^0}^{t}B\leq v_2 e$ where, $v_2=x^{0^{t}}By^0$.
\\Now from the assumption we have $y^0 > 0$ coordinatewise. Hence $x^0B=v_2e^t$ as otherwise,
$$v_2 = {x^0}^{t}By^0 < v_2e^ty_0=v_2$$
is a contradiction. Hence the Lemma follows.
 \end{proof}
A corresponding non-zerosum version of Lemma \ref{lemma1} is provided below.
\begin{lemma}\label{lemma6}
Assume player-2 controlled transition. Suppose, 
$\exists \; \beta_0 \in [0,1)$ and NE $(f^{\beta},g^0)$ in $\Gamma_{\beta}$ such that $g^0$ is completely mixed,
in \textbf{every} $\beta$-discounted non-zerosum stochastic game for all $\beta > \beta_0$.
Let, $\beta_n \in [\beta_0,1)$ be such that $\beta_n \uparrow 1$. Let $(f_n,g^0)$ be NE for $\beta_n$-discounted non-zerosum stochastic game. Suppose $f_n \rightarrow f_0$ coordinate-wise, that 
is $f_n(s) \rightarrow f_0(s)$ for each state $s\in S$, then, $(f_0,g^0)$ is equilibrium pair in the undiscounted non-zerosum stochastic game $\Gamma$.
\end{lemma}
\begin{proof}
Under one player (player-2) controlled transition probability, an undiscounted non-zerosum stochastic game has value restricted to stationary strategy (CITATION REQUIRED). Suppose, $(f_n,g^0)$ is NE for $\beta_n$-discounted stochastic game. As we have a completely mixed strategy $g_0$ for player-2 in the $\beta_n$-discounted stochastic game from Theorem \ref{theorem6} and Lemma \ref{lemma5} we have
$$I^2_{\beta_n}(f_n,g) \equiv v^2_{\beta_n} $$
coordinatewise for any stationary strategy $g$ of player-2, where
\[I^2_{\beta_n}(f_n,g)=(I^2_{\beta_n}(f_n,g)(s_1),\cdots, I^2_{\beta_n}(f_n,g)(s_K))^t\] and
$v^2_{\beta_n}=(v^2_{\beta_n}(s_1),\cdots,v^2_{\beta_n}(s_K))$. Therefore we have $[I-\beta_n Q(g)]^{-1}$ $
r^2(f_n,g)\equiv v^2_{\beta_n}$. Now we have $[I-\beta_n Q(g)]^{-1} =\sum_{k=0}^\infty\beta_n^k Q^k(g)$  is a non-negative matrix. Hence we have the expression of $r^2(f_n,g)$ as follows.
$$ r^2(f_n,g) \equiv [I-\beta_n Q(g)] v^2_{\beta_n} $$
Now we have $f_n \rightarrow f_0$ point-wise and the reward function $r(.,.)$ is a continuous function on the strategy of player-1. Hence, for any given $\epsilon > 0 $ we have,
$$[I-\beta_n Q(g)] v^2_{\beta_n} - \epsilon e \leq r^2(f_0,g) \leq [I-\beta_n Q(g)] v^2_{\beta_n}+ \epsilon e$$
coordinate-wise for all $n\geq N_0$, for some $N_0$. Where $e$ is a suitable length column vector with all entry as 1. Therefore we have,
$$v^2_{\beta_n} - \epsilon [I-\beta_n Q(g)]^{-1} e \leq [I-\beta_n Q(g)]^{-1} r^2(f_0,g) \leq  v^2_{\beta_n}+ \epsilon [I-\beta_n Q(g)]^{-1} e.$$
As, $(1-\beta_n)$ is always non-negative for all $\beta_n \in [0,1) $, we have,
$$(1-\beta_n)v^2_{\beta_n} - (1-\beta_n) \epsilon [I-\beta_n Q(g)]^{-1} e \leq (1-\beta_n) [I-\beta_n Q(g)]^{-1} r^2(f_0,g) $$$$\leq (1-\beta_n) v^2_{\beta_n}+ (1-\beta_n) \epsilon [I-\beta_n Q(g)]^{-1} e.$$
Since $[I-\beta_n Q(g)]^{-1}e= [\sum_{k=0}^{\infty} \beta_n^k Q^k(g)]e = \sum_{k=0}^{\infty} \beta_n^k [Q^k(g)e] = \sum_{k=0}^{\infty} \beta_n^k e$ as $Q^k$ is a stochastic matrix for each $k$. Therefore the above inequality reduces to the following inequality.
$$(1-\beta_n)v^2_{\beta_n} -  \epsilon  e \leq (1-\beta_n) [I-\beta_n Q(g)]^{-1} r^2(f_0,g) \leq (1-\beta_n) v^2_{\beta_n}+ \epsilon e.$$
Now if we let $\beta_n \uparrow 1$ from \cite{tp1981} the above inequality will look as follows.
$$v^2-\epsilon e \leq \Phi^2(f_0,g) \leq v^2+\epsilon e. $$ 
Where, $v^2=(v^2(s_1),\cdots,v^2(s_K))^t$ is the value of the undiscounted stochastic game. This is true for any $\epsilon > 0 $. Hence $(f_0,g^0)$ constricted above is NE in the non-zerosum undiscounted stochastic game.
 \end{proof}

The following is an example of a non-zerosum completely mixed undiscounted stochastic game. This is the corresponding non-zerosum version of Example \ref{example3}.
\begin{example}\label{example5}
$$R(s_1) =
\begin{bmatrix}
    0,1/(0,1)      & 2,-1/(1,0)   \\
    2,-1/(0,1)       & 0,1/(1,0) 
\end{bmatrix}      ,   \qquad    R(s_2) =
\begin{bmatrix}
    4,-3/(1,0)      & -2,3/(0,1)   \\
    -2,3/(1,0)       & 4,-3/(0,1) 
\end{bmatrix}$$
This example is a player-2 controlled stochastic game with states $s_1$ and $s_2$. player-2 controlled the transition. In both state $s_1$ and $s_2$ if player-2 chooses action 1 (column 1) then the game moves to state $s_1$ in the next day and if player-2 chooses action 2 (column 2) then the game moves to state $s_2$ in the next day.

\par The unique equilibrium strategy for player-1 is $\{(\frac{1}{2},\frac{1}{2}), (\frac{1}{2},\frac{1}{2})\}$ in the above-mentioned game. that is choosing row 1 and row 2 is state $s_1$ with probability $\frac{1}{2}$ and $\frac{1}{2}$ respectively and choosing row 1 and row 2 is state $s_2$ with probability $\frac{1}{2}$ and $\frac{1}{2}$ respectively. The unique equilibrium strategy for player-2 is $\{(\frac{1}{2},\frac{1}{2}), (\frac{1}{2},\frac{1}{2})\}$ for the above-mentioned game.

In example \ref{example5}, we can see that for all $\beta \in [0,1)$ the $\beta$-discounted non-zerosum stochastic game $\Gamma_{\beta}$ is completely mixed.

$$B_{\beta}^1(s_1) =
\begin{bmatrix}
    \beta & 2-\beta   \\
    2-\beta       & \beta 
\end{bmatrix}      ,      B_{\beta}^2(s_1) =
\begin{bmatrix}
    1-\beta      & \beta-1   \\
    \beta-1       & 1-\beta 
\end{bmatrix}$$

$$B_{\beta}^1(s_2) =
\begin{bmatrix}
    4-3\beta & 3\beta-2   \\
    3\beta-2       & 4-3\beta 
\end{bmatrix}      ,      B_{\beta}^2(s_2) =
\begin{bmatrix}
    3(\beta-1)      & 3(1-\beta)   \\
    3(1-\beta)      & 3(\beta-1)
\end{bmatrix}$$

\end{example}
The following Lemma is required to prove the later Theorem stating equalizer rule for general undiscounted stochastic game. In the following Lemma and Theorem we do not anymore have the assumption that only one player controls the transition probability. 
\begin{lemma}\label{lemma7}
Let $\Gamma$ be a non-zerosum undiscounted stochastic game (player-2 controlled is not necessary for this result). Let $(f^0,g^0)$ be a completely mixed NE for the undiscounted stochastic game $\Gamma$. Then there exists two vectors $\gamma^1=(\gamma^1(s_1),\cdots,\gamma^1(s_K))^t$ and $\gamma^2=(\gamma^2(s_1),\cdots,\gamma^2(s_K))^t$ such that
$$v_{f^0,g^0}^2(s)+\gamma^2(s) =  r^2(f^0,j,s)+  \sum_{s'\in S} q(s'|s,f^0,j)\gamma^2(s')$$
for all pure action $j$ for player-2 and $s\in S$. And
  		
$$v_{f^0,g^0}^1(s)+\gamma^1(s) =  r^1(i,g^0,s)+  \sum_{s'\in S} q(s'|s,i,g^0)\gamma^1(s')$$
for all pure action $i$ of player-1, and $s\in S$,
where, $v_{f^0,g^0}^k(s)$ for $k=\{1,2\}$ is the value in the undiscounted stochastic game for player $k$ corresponding to the NE $(f^0,g^0)$ in state $s \in S$.
\end{lemma}
\begin{proof}
From Markov decision process (\cite{fillerbook1996}, theorem 3.8.4) we have existence of vectors $\gamma^1$ and $\gamma^2$ such that $$v_{f^0,g^0}^2(s)+\gamma^2(s) = \max_{\sigma} \bigg[ r^2(f^0,\sigma,s)+  \sum_{s'\in S} q(s'|s,f^0,\sigma)\gamma^2(s') \bigg],$$
for all  stationary strategy $\sigma\in \mathbb{P}_2$ and $s\in S$. And
$$v_{f^0,g^0}^1(s)+\gamma^1(s) = \max_{\mu} \bigg[ r^1(\mu,g^0,s)+  \sum_{s'\in S} q(s'|s,
\mu,g^0)\gamma^1(s') \bigg],$$
for all stationary strategy $\mu \in \mathbb{P}_1$ and $s\in S$. Therefore for all pure action $i$ for player-1, for all pure action $j$ for player-2 and $s\in S$: 
$$v_{f^0,g^0}^2(s)+\gamma^2(s) \geq  r^2(f^0,j,s)+  \sum_{s'\in S} q(s'|s,f^0,j)\gamma^2(s') \qquad \text{ and, }$$
$$v_{f^0,g^0}^1(s)+\gamma^1(s) \geq r^1(i,g^0,s)+ 
\sum_{s'\in S}q(s'|s,i,g^0)\gamma^1(s').$$
If possible let us assume, we have strict inequality for some pure action $j_0$ of player-2. That is-
$$v_{f^0,g^0}^2(s)+\gamma^2(s) >  r^2(f^0,j_0,s)+  \sum_{s'\in S} q(s'|s,f^0,j_0)\gamma^2(s').$$
Therefore,
\[v_{f^0,g^0}^2(s)+\gamma^2(s)=v_{f^0,g^0}^2(s)\sum_j g^0_j(s)+\gamma^2(s)\sum_j g^0_j(s)\]
\[=\sum_j g^0_j(s) [v_{f^0,g^0}^2(s)+\gamma^2(s)] > \sum_j g^0_j(s)[r^2(f^0,j,s) +  \sum_{s'\in S} q(s'|s,f^0,j) \gamma^2(s')]\]
\[ = \sum_j r^2(f^0,j,s)g^0_j(s) + \sum_{s'\in S} \sum_j q(s'|s,f^0,j)g^0_j(s) \gamma^2(s')\]
\[=  r^2(f^0,g^0,s)+  \sum_{s'\in S} q(s'|s,f^0,g^0)\gamma^2(s') = v_{f^0,g^0}^2(s) +\gamma^2(s),\]
which is a contradiction. Hence equality holds for all pure stationary strategy $j$ for player-2. Exactly in the similar way we can show equality for pure action of player-1 also.
 \end{proof}
Using \ref{lemma7} we can show that for an undiscounted stochastic game with one completely mixed NE the game process equalizer rule for undiscounted payoff. Unlike the zerosum case (see Theorem \ref{theorem5}) we have equalizer for payoff for both players. 
\begin{theorem}\label{theorem7}
Let $\Gamma$ be a non-zerosum undiscounted stochastic game. NOT necessarily single player controlled transition. Also assume that, the undiscounted stochastic game is completely mixed. Let $(f^0,g^0)$ be a NE in the game. Then we have the following equalizer rule, 
$$\Phi^{(2)}(f^0, g^0)=\Phi^{(2)}(f^0, g)$$
$$\Phi^{(1)}(f^0, g^0)=\Phi^{(1)}(f, g^0)$$ for all  stationary strategy $f\in \mathbb{P}_{1}$ and  stationary strategy $g\in \mathbb{P}_{2}$.
\end{theorem}
\begin{proof}
Using Remark \ref{remark1} we have $(f^0,g^0)_c$ is $(f^0,g^0)$ restricted to states $s\in C_c$ is a NE in the restricted game $\Gamma_c$. Since the states $C_c$ are irreducible, the value for both player-1 and player-2 corresponding to NE $(f^0,g^0)_c)$ is independent of states (say $v^1_c$ and $v_c^2$ respectively). Therefore there exists vector $\gamma^2$ such that from Lemma \ref{lemma7}:
$$v_c^2+ \gamma^2(s)= r_c^2(f^0,j,s)+  \sum_{s'\in S} q(s'|s,f^0,j)\gamma^2(s')$$
where, $j$ is any pure strategy of player-2's and $s\in C_c$. Now for any stationary strategy $g$ of player-2 we have the following equality.
$$ v_c^2 g_j(s)+ \gamma^2(s)g_j(s)= r_c^2(f^0,j,s)g_j(s)+  \sum_{s'\in S} q(s'|s,f^0,j)g_j(s)\gamma^2(s')$$
where, $j$ is any pure strategy of player-2's and $s\in C_c$. This implies,
$$ \sum_{j} v_c^2g_j(s)+ \sum_{j}\gamma^2(s)g_j(s)=\sum_{j} r_c^2(f^0,j,s)g_j(s)+ \sum_{j} \sum_{s'\in S} q(s'|s,f^0,j)g_j(s)\gamma^2(s')$$
for all $s\in C_c$. Then it follows:
$$\text{for all }s\in C_c; \qquad v_c^2+ \gamma^2(s)= r_c^2(f^0,g,s)+  \sum_{s' \in S} q(s'|s,f^0,g)\gamma^2(s').$$
This then implies,
$$ v_c^2\textbf{1}+ \gamma^2= r_c^2(f^0,g)+  Q(f^0,g)\gamma^2.$$
Multiplying be the Markov matrix $Q^*(f^0,g)$ in both the sides of the above equality we obtain:
$$ v_c^2\textbf{1}+ Q^*(f^0,g)\gamma^2= \Phi_c^2(f^0,g)+  Q^*(f^0,g)\gamma^2.$$
Therefore we indeed have:
$$\Phi_c^2(f^0,g^0) =\Phi_c^2(f^0,g).$$
\\ The above argument is true for all the individual completely mixed games $\Gamma_1,\Gamma_2,\cdots\Gamma_k$. Define $p_g(s,s')$ be the probability that starting in state $s\in H$, the first state reached outside $H$ is $s'$. So we have,
$$\Phi^2(f^0,g)(s)= \sum_{s\in H^c} p_g(s,s') \Phi^2(f^0,g)(s')  ,$$
where, $s\in H$. We already have $\Phi_c^2(f^0,g) =\Phi_c^2(f^0,g^0)$ for $s\in H^c$.
\\ For other side of the equality is similar with the proof for player-1. Using Lemma \ref{lemma7} we can derive the following for each $\Gamma_c$. 

$$v_c^1(s)+\gamma^1(s) =  r^1(f,g^0,s)+  \sum_{s' \in S} q(s'|s,f,g^0)\gamma^1(s')$$
for all $f\in \mathbb{P}_1$. Now multiplying the Markov matrix $Q^*(f,g^0)$ on both the sides we have the desired equality.
 \end{proof}

\begin{theorem}\label{theorem8}
Let $\Gamma$ be a non-zerosum player-2 controlled undiscounted stochastic game. Let $(f^0,g^0)$ be a NE with $g^0$ as an uniformly discount equilibrium. Then

$$v^1_{f^0,g^0}(s) = \sum_{s'}q(s'|s,g^0) v^1_{f^0,g^0} (s')$$
$$v^2_{f^0,g^0}(s) = \sum_{s'}q(s'|s,g^0) v^2_{f^0,g^0} (s')$$

\end{theorem}
\begin{proof}
The proof is easy. So we are skipping the proof.
 \end{proof}


\begin{remark}\label{remark3}

For the undiscounted player-2 controlled non-zerosum stochastic $\Gamma$ game the existence of uniform discount equilibrium $(f^0,g^0)$ is shown by \cite{tp1981}. Also it is shown that $(f^{\beta},g^0)$ with $f^0(s)= \lim_{\beta \uparrow 1}f^{\beta}(s)$
 is a NE in the $\beta$-discounted stochastic game $\Gamma_{\beta}$ for all $\beta > \beta_0$. If the Undiscounted non-zerosum stochastic game is completely mixed then the uniformly discount equilibrium $(f^0,g^0)$ is completely mixed hence for $\beta$ closed to $1$\, $(f^{\beta},g^0)$ is completely mixed NE in $\Gamma_{\beta}$. Now from Theorem \ref{theorem6} we have $(f^{\beta}(s),g^0(s))$ is a NE in the bimatrix game $(B_{\beta}^1(s),B_{\beta}^2(s))$. So we have $B_{\beta}^1(s)g^0(s)=c\textbf{e}$. Therefore from Lemma 4.1 of \cite{tp1981} we have $R^1(s)g^0(s)=c_1 \textbf{e}$. It is also mentioned in \cite{tp1981} that for completely mixed games $R^1(s)$ and $R^2(s)$ are square matrix.
\end{remark}

\begin{proposition}
Let $\Gamma$ be non-zerosum player-2 controlled undiscounted completely mixed stochastic game. Then the payoff matrices $R^2(s)$ for all $s\in S$ are non-singular.
\end{proposition}
\begin{proof}

In \cite{tp1981} they showed existence of an uniform discount NE $(f^0,g^0)$ in the undiscounted $\Gamma$
such that $f^0(s)=\lim_{\beta \uparrow 1}
f^{\beta}(s)$ for all $s \in S$ and $(f^{\beta},g^0)$
is NE in the discounted game $\Gamma_{\beta}$ for all $\beta>\beta_0$. 
Without loss of generality assume $r^k(i,j,s)>0$ for all $k=\{1,2\}$ and for all $s\in S$. Now since $\Gamma$ is completely mixed $(f^0,g^0)$ is also completely mixed. Hence  $(f^{\beta},g^0)$ is also completely mixed for all $\beta>\beta_1$.
If possible let us assume for some $s_0, R^2(s_0) $ is singular. we already have from Theorem \ref{theorem6}-
\begin{center}
    $B^1_{\beta}(s_0)g^0(s_0)=v^1_{f^{\beta},g^0}\textbf{1}$ ,\\
    $f^{\beta}(s_0)B^2_{\beta}(s_0)=v^2_{f^{\beta},g^0}\textbf{1}^t$
\end{center}
Using Lemma 4.1; \cite{tp1981} we can conclude $B^2_{\beta}(s_0)$ is also singular. Hence have some $f_*^{\beta}\in \mathbb{P}_2$ with $f_*^{\beta}(s_0)B^2_{\beta}(s_0)=v^2_{f^{\beta},g^0}\textbf{1}^t$. $(f_*^{\beta},g^0)$ is another NE in $\Gamma_{\beta}$. From Lemma 5.2 \cite{tp1981} we have $S(g^0)=\{f|(f,g^0)$ is NE in $\Gamma_{\beta}\} $ is convex. Hence we can choose $\lambda$ such that $(\lambda f_*^{\beta} +(1-\lambda)f^{\beta},g^0)$ is a non-completely mixed equilibrium. Now a sub sequence of $(1+\lambda) f_*^{\beta} +\lambda f^{\beta}$ will converge to some $\bar{f}$ such that $(\bar{f},g^0)$ will be a non-completely mixed NE (from Lemma \ref{lemma6}) in $\Gamma$. Which is a contradiction and hence the proposition follows.
 \end{proof}

\begin{proposition}\label{prop2}
Let $\Gamma$ be non-zerosum player-2 controlled undiscounted completely mixed stochastic game. If $R^1(s)$ is non-singular for all $s\in S$. Then $T=\{g|(f,g)$ is a NE in $\Gamma$ for some $f 
\in \mathbb{P}_1 \}$ is singleton.
\end{proposition}
\begin{proof} The proof is on the 
same way of filar's argument(\cite{filr1985} proposition 3.4) for zero-sum games. In 
\cite{tp1981} (page 390) provide the existence of 
uniform discount equilibrium in $\Gamma$. 
Let $(f^0,g^0)$ be NE in $\Gamma$ with $g^0$ uniform discount equilibrium. Let we have some other NE $(f^*,g^*)$
(Note: $f^*$ may be equal with $f^0$). From remark 1 the sets $C_1,\cdots,C_k$ and $H$ are same for $g^0$ 
and $g^*$. Also for all $s\in H$ the number of action available in $s$ is exactly $1$. So $g^0(s)=g^*(s)$ for all $s\in H$. Now it is enough to consider the 
subgames $\Gamma_1, \cdots ,\Gamma_k$ separately. Without loss of generality we consider only $\Gamma_1$ and assume
that $C_1$ has $S_1$ states. Also, associate with a 
NE $(f,g)$, the value $v^k = v^k(s)$ is a constant for all $s \in C_1$ and $k\in \{1,2\}$. The 
stationary matrices $Q^*(g^0)$ and $Q^*(g^*)$ each have identical rows $u^0 =
(u^0(l),\cdots, u^0(S_1))$ and $u^* = (u^*(1),\cdots, u^*(S_1))$ with $u^0(s)$ and $u^*(s) > 0$ for all
$s \in C_1$. By theorem 8, for every pure stationary strategy a for player-1 in $\Gamma_1$ we
have, for every $s \in C_1$,
$$v^1_{(f^0,g^0)}=\Phi^1(\sigma, g^0)(s) = [Q^*(g^0)r^1(\sigma, g) ]_s$$
$$=\sum_{s'=1}^{S_1} u^0(s')[\sigma(s')R^1(s') g^0(s')]$$
\begin{equation}
=\sum_{s' =1}^{ S_1} \sum_{j=1}^{n_{s'}}
 r^1(\sigma, j,s')u_j^0(s')
\end{equation}

where $u_j^0(s') = u^0(s')g^0_j(s')$ for all $j = 1,\cdots, n_s$, and $s' \in C_1$. The above equality holds
with $g^*$ in place of $g^0$ and with $u^*(s') = u^*(s')g^*(s')$ in place of $u^0(s')$. Now let
$z_j(s) = u^0_j(s) - u^*_j(s)$ for all $j = 1,\cdots, n_s$ and $s \in C_1$. Then from above equation we obtain
$$\sum_{s' =1}^{ S_1} \sum_{j=1}^{n_{s'}}
r^1(\sigma, j,s')z_j(s') =v^1_{(f^0,g^0)}-v^1_{(f^*,g^*)}=\bar{c} $$
for every pure stationary strategy $\sigma$ for I in $\Gamma_1$. Let
$$ Z =(Z_l:Z_2:\cdots:Z_{S_1})^t$$
be a column vector such that $Z_s = (Z_1(s), Z_2(s), \cdots , Z_n(S))$ for each $s \in C_1$, and let
$t_1$, be the number of pure stationary strategies for player-1 in $\Gamma_1$. Fix $\sigma(s)$ for each $s > 1$
and consider the $n_1$ equations extracted from (1) by letting $\sigma(1)$ range over the $n_1$-dimensional unit basis vectors (Since $R^1(s)$ is square from remark 3). These are equivalent to $$R^1(1)Z_l = \alpha\textbf{1}.$$ Now, if $\alpha = 0$ the argument follows exactly in the same way of filar's argument for zero sum games with help of Assumption. 
If $\alpha \neq 0$ the argument goes exactly in the same way of Filar's argument with the help of Remark 3.
  \end{proof}

\begin{proposition}\label{prop3}
Let $\Gamma$ be non-zerosum player-2 controlled undiscounted completely mixed stochastic game. Then given a NE $(f^*,g^*)$, $S(g^*)=\{f|(f,g^*)$ is a NE $\}$ is singleton.
\end{proposition}
\begin{proof}
From lemma 5.2 of \cite{tp1981} we already have $S(g^*)$ is convex. Since we have $\Gamma$ is completely mixed
all elements of $S(g^*)$ has to be completely mixed. Assume we have $f^0,f^* \in S(g^*)$. Therefore we can choose $\lambda$ such the $\lambda f^0(s)+(1-\lambda)f^*(s)$ is not completely mixed for some $s\in S$. Which is a contradiction. Hence the proposition is true.
 \end{proof}
For non-zerosum stochastic games completely mixed game does not imply unique Nash Equilibrium. But under some nonsingularity assumption of the payoff matrix the undiscounted game process unique Nash equiluibrium. 
\begin{theorem}\label{theorem9}
The two person non-zerosum player-2 controlled undiscounted stochastic game $\Gamma$ is completely mixed. If $R^1(s)$ is nonsingular for all $s\in S$. Then the game process an unique Nash Equilibrium. 
\end{theorem}

\begin{proof}
Let $(f^0,g^0)$ and $(f^*,g^*)$ be two NE is $\Gamma$. Then from Theorem \ref{theorem8} we have $(f^0,g^*)$ and $(f^*,g^0)$ are also NE of the stochastic game $\Gamma$. Which contradicts the Proposition \ref{prop2} and \ref{prop3}. Hence, $\Gamma$
process an unique NE. 
 
 \end{proof}

\begin{theorem}
Let $\Gamma_{\beta}$ be a non-zerosum $\beta$-discounted stochastic game. The game is \textit{not} necessarily single player controlled. Also assume that the $\beta$-discounted stochastic game has a completely mixed NE $(f^0,g^0)$. Then we have the following equalizer rules, 
$$ I_{\beta}^{(2)}(f^0, g^0)=I_{\beta}^{(2)}(f^0, g)\quad \text{and},$$
$$ I_{\beta}^{(1)}(f^0, g^0)= I_{\beta}^{(1)}(f, g^0)$$ for all stationary strategy $f\in \mathbb{P}_{A_1}$ and stationary strategy $g\in \mathbb{P}_{A_2}$.
\end{theorem}
\begin{proof}

From markov decision process we have the following results-
$$v_{\beta}^1(s) = \max_{\mu \in \mathbb{P}_1}\Bigg[ r^1(\mu,g^0,s)+ \beta \sum_{s' \in S} q(s'|s,\mu,g^0)v_{\beta}^1(s')\bigg]$$
$$v_{\beta}^2(s) = \max_{\sigma\in \mathbb{P}_2}\Bigg[ r^2(f^0,\sigma,s)+ \beta \sum_{s' \in S} q(s'|s,f^0,\sigma)v_{\beta}^2(s')\Bigg]$$

Now using the similar technique as used in Lemma \ref{lemma7} we can show that equality holds for all pure strategy (hence for all strategies) of player-1 and player-2.
Therefore we have 
$$v_{\beta}^1(s) =  r^1(f,g^0,s)+ \beta \sum_{s' \in S} q(s'|s,f,g^0)v_{\beta}^1(s')$$
$$v_{\beta}^2(s) = r^2(f^0,g,s)+ \beta \sum_{s' \in S} q(s'|s,f^0,g)v^2_{\beta}(s')$$
for all $f\in \mathbb{P}_1$ and $g \in \mathbb{P}_2$.
Writing the last equation in a vector from we  have
$$v_{\beta}^2 = r^2(f^0,g)+ \beta Q(f^0,g)v_{\beta}^2$$
Now pre-multiplying the equation by $[I-\beta Q(f^0,g)]^{-1}$ we get-
$$[I-\beta Q(f^0,g)]^{-1}v_{\beta}^2 = [I-\beta Q(f^0,g)]^{-1}r^2(f^0,g)+ \beta [I-\beta Q(f^0,g)]^{-1} Q(f^0,g)v_{\beta}^2.$$
This implies,
$$[I-\beta Q(f^0,g)]^{-1}v_{\beta}^2 = [I-\beta Q(f^0,g)]^{-1}r^2(f^0,g)+ \beta [\sum_{k=0}^{\infty}\beta^kQ^k(f^0,g)] Q(f^0,g)v_{\beta}^2.$$
$$\implies[I-\beta Q(f^0,g)]^{-1}v_{\beta}^2 = [I-\beta Q(f^0,g)]^{-1}r^2(f^0,g)+  [\sum_{k=1}^{\infty}\beta^k Q^k(f^0,g)]v_{\beta}^2.$$
$$\implies[I-\beta Q(f^0,g)]^{-1}v_{\beta}^2 =I_{\beta}^2(f^0,g)+  [-I+\sum_{k=0}^{\infty}\beta^k Q^k(f^0,g)]v_{\beta}^2.$$
So finally we have for all 
$$\implies v_{\beta}^2=I_{\beta}^2(f^0,g)$$
for all $g \in \mathbb{P}_2$.
Similarly we can show  $I_{\beta}^{(1)}(f^0, g^0)= I_{\beta}^{(1)}(f, g^0)$ for all $f\in \mathbb{P}_{1}$. 
 \end{proof}

\section{Open Problem:} 
In a switching control Stochastic game (\cite{filr1981}) we get a partition $S_1$ and $S_2$ of state space $S$, such that in any states of $S_1$ player-1 alone controls the transition probability and in any states of $S_2$ player-2 alone controls the transition probability. The transition probabilities are as follows.
\begin{center}
$q(s'|s,i,j)=q(s'|s,i)$ for all $s\in S, s'\in S_1 i\in A_1$ and $j\in A_2.$\\
$q(s'|s,i,j)=q(s'|s,j)$ for all $s\in S, s'\in S_2 i\in A_1$ and $j\in A_2.$

\end{center}

It is not known whether similar results of Theorem \ref{main.theorem} and Theorem \ref{theorem2} holds for Switching control undiscounted stochastic games. 

\section{Acknowledgement}
We would like to thanks the anonymous referee for there valuable and delicate commands that helped to structure the paper better.

\bibliographystyle{siam}\bibliography{ref} 

\begin{thebibliography}{10}

\bibitem{sujatha2016}
{\sc S.~Babu, N.~Krishnamurthy, and T.~Parthasarathy}, {\em Stationary,
  completely mixed and symmetric optimal and equilibrium strategies in
  stochastic games}, International Journal of Game Theory,  (2016), pp.~1--22.

\bibitem{black1968}
{\sc D.~Blackwell and T.~S. Ferguson}, {\em The big match}, The Annals of
  Mathematical Statistics, Vol. 39 (1968), pp.~159--163.

\bibitem{fillerbook1996}
{\sc J.~Filar and K.~Vrieze}, {\em Competitive Markov Decision Processes},
  Springer, 1996.

\bibitem{filr1981}
{\sc J.~A. Filar}, {\em Ordered field property for stochastic games when the
  player who controls transitions changes from state to state}, Journal of
  Optimization Theory and Applications, Vol. 34 (1981), pp.~501--515.

\bibitem{filr1985}
\leavevmode\vrule height 2pt depth -1.6pt width 23pt, {\em The completely mixed
  single-controller stochastic game}, Proceedings of the American Mathematical
  Society, Vol. 95 (1985), pp.~585--594.

\bibitem{rag1984}
{\sc J.~A. Filar and T.~E.~S. Raghavan}, {\em A matrix game solution to single
  controller stochastic game.}, Mathematics of Operations Research, Vol. 9
  (1984), pp.~356 -- 362.

\bibitem{fink64}
{\sc A.~M. Fink}, {\em “equilibrium in a stochastic n-person game}, J.
  Science of Hiroshima University, Vol. 28 (1964), pp.~89--93.

\bibitem{gillet}
{\sc D.~Gillette}, {\em Stochastic games with zero stop probabilities},
  Princeton University Press,  (1957).

\bibitem{Hordijk1979}
{\sc A.~Hordijk and L.~C.~M. Kallenberg}, {\em Linear programming and markov
  decision chains}, Management Science, Vol. 25 (1979), pp.~352 -- 362.

\bibitem{kapl1945}
{\sc I.~Kaplansky}, {\em A contribution to von neumann's theory of games},
  Annals of Mathematics Second Series, Vol. 46 (1945), pp.~474--479.

\bibitem{nowak93}
{\sc A.~S. Nowak and T.~E.~S. Raghavan}, {\em A finite step algorithm via
  bimatrix games to a single controller nonzero-sum stochastic games},
  Mathematics Programming, Vol. 59 (1993), pp.~249--259.

\bibitem{tp1981}
{\sc T.~Parthasarathy and T.~E.~S. Raghavan}, {\em An orderfield property for
  stochastic games when one player controls transition probabilities}, Journal
  of Optimization Theory and Applications, VOL. 33 (1981), p.~375–392.

\bibitem{Rogers1969}
{\sc P.~D. Rogers}, {\em Nonzero-sum stochastic games}, PhD thesis, University
  of California, Berkeley. Report ORC 69-8.,  (1969).

\bibitem{shaply53}
{\sc L.~S. Shapley}, {\em Stochastic games}, Proceedings of the National
  Academy of Sciences of the United States of America, VOL. 39 (1953),
  pp.~1095--1100.

\bibitem{Sobel1971}
{\sc M.~Sobel}, {\em Noncooperative stochastic games}, Annals Mathematical
  Statistics, Vol. 42 (1971), pp.~1930--1935.

\bibitem{takahashi64}
{\sc M.~Takahashi}, {\em Equilibrium points of stochastic noncooperative
  n-person games}, J. Science of Hiroshima University, Vol. 28 (1964),
  pp.~95--99.

\end{thebibliography}

\end{document}